\documentclass[12pt]{article}

\usepackage{amsmath}
\usepackage{amsthm}
\usepackage{amssymb}
\usepackage{amsfonts}
\usepackage{latexsym}
\usepackage{graphicx}
\usepackage{avant}
\usepackage{epsfig}
 \numberwithin{equation}{section}
 
  \DeclareMathOperator{\vol}{vol}
    \DeclareMathOperator{\RR}{\mathbb{R}}
    \DeclareMathOperator{\ZZ}{\mathbb{Z}}
    \DeclareMathOperator{\CC}{\mathbb{C}}
    \DeclareMathOperator{\heis}{\mbox{Heis}}
    \DeclareMathOperator{\imm}{\mbox{Im}}
    \DeclareMathOperator{\ree}{\mbox{Re}}

 \newtheorem{thm}{Theorem}[section]
 
 \newtheorem{lem}[thm]{Lemma}
 \newtheorem{prop}[thm]{Proposition}
 \newtheorem{con}[thm]{Conjecture}
 \theoremstyle{definition}
 
 \theoremstyle{remark}
 
 \numberwithin{equation}{section}

\title{Spectral Asymptotics Revisited}
\author{Robert S. Strichartz\footnote{Research supported in part by the National Science Foundation, grant DMS-0652440}}
\date{}

\begin{document}
\maketitle
\begin{large}
\begin{centering}
Department of Mathematics\\
Malott Hall\\
Cornell University\\
Ithaca, NY 14853\\
\texttt{str@math.cornell.edu}\\
\end{centering}
\end{large}

\begin{abstract}
We discuss two heuristic ideas concerning the spectrum of a Laplacian, and we give theorems and conjectures from the realms of manifolds, graphs and fractals that validate these heuristics. The first heuristic concerns Laplacians that do not have discrete spectra: here we discuss the notion of ``spectral mass'' closely related to the ``integrated density of states'', an average of the diagonal of the kernel of the spectral projection operator, and show that this can serve as a substitute for the eigenvalue counting function. The second heuristic is an ``asymptotic splitting law'' that describes the proportions of the spectrum that transforms according to the irreducible representations of a finite group that acts as a symmetry group of the Laplacian. For this to be valid we require the existence of a fundamental domain with relatively small boundary. We also give a version in the case that the symmetry groups is a compact Lie group. Many of our results are reformulations of known results, and some are merely conjectures, but there is something to be gained by looking at them together with a unified perspective. 
\end{abstract}

\section{Introduction}

What is ``spectral asymptotics'' if not the study of the asymptotic behavior as $\lambda\to\infty$ of the eigenvalue counting function

\begin{equation}
N(\lambda)=\#\{\lambda_{j}\leq\lambda\}
\end{equation} of some operator? For this to make sense the operator must have a discrete spectrum. In this paper we restrict attention to operators that may be called ``Laplacians'', but the classical theory [RS] is not restricted to Laplacians, and many of the ideas in this paper can be extended at least to Schr\"odinger operators (Laplacian plus potential) and perhaps beyond. But what if the operator does not have a discrete spectrum? Must we give up on the idea of spectral asymptotics? Not at all!

We will discuss a notion of ``spectral mass'' $M(\lambda)$ that can replace $N(\lambda)$. To motivate the definition recall the famous Weyl asymptotic law

\begin{equation}\label{eq:weyl}
N(\lambda)=c_{n}\vol(\Omega)\lambda^{n/2}+o(\lambda^{n/2})\text{ as }\lambda\to\infty
\end{equation} when the operator is $-\Delta$ on a compact domain $\Omega$ in $\RR^{n}$ (or more generally a Riemannian manifold of dimenson $n$) with smooth boundary with Dirichlet or Newmann boundary conditions. (Here $c_{n}$ is an explicit dimensional constant.) We divide~(\ref{eq:weyl}) by $\vol(\Omega)$, and note that $N(\lambda)/\vol(\Omega)$ is equal to the average value of $K_{\lambda}(x,x)$ on $\Omega$, where $K_{\lambda}$ is the kernel of the spectral projection operator $E_{\lambda}$ onto the $[0,\lambda]$ portion of the spectrum,
\begin{equation}\label{eq:overline}
K_{\lambda}(x,y)=\sum_{\lambda_{j}\leq \lambda}\varphi_{j}(x)\overline{\varphi_{j}(x)}
\end{equation} where $\{\varphi_{j}\}$ is an orthonormal basis of eigenfunctions with associated eigenvalues $\{\lambda_{j}\}$. So in this case we define the spectral mass function
\begin{equation}
M(\lambda)=\frac{1}{\vol(\Omega)}\int_{\Omega}K_{\lambda}(x,x)d\mu(x),
\end{equation}($\mu$ is Lebesgue measure or the Riemannian mesure on $\Omega$). Our simple observation above is that 
\begin{equation}
M(\lambda)=\frac{N(\lambda)}{\vol(\Omega)},
\end{equation} and so Weyl's asymptotic law says 
\begin{equation}\label{eq:M}
M(\lambda)=c_{n}\lambda^{n/2}+o(\lambda^{n/2}).
\end{equation}

Now  consider a Laplacian so that $-\Delta$ is a nonnegative self-adjoint operator on $L^{2}(\Omega,\mu)$ where the measure is not finite. The spectral theorem still supplies the spectral projection operators $E_{\lambda}$, and in many cases these are given by integration against a kernel $K_{\lambda}(x,y)$, but of course~(\ref{eq:overline}) is not valid. Our first heuristic idea is to try to define $M(\lambda)$ as some sort of average value of $K_{\lambda}(x,x)$ and to study its asymptotic behavior. Note that we do \emph{not} expect~(\ref{eq:M}) to hold, and indeed in many cases the asymptotic behavior will be as $\lambda\to0$ or some other natural value. Also, in some cases we may have to settle for upper and lower ($M^{+}(\lambda)$, $M^{-}(\lambda)$) spectral masses if we can't prove that a true average exists. This is clearly a heuristic idea because it is necessary to decide on the meaning of ``average'' in each context.

We should point out that the same idea has been used in quantum mechanics under the name \emph{integrated density of states}. The usual context is the Schr\"odinger operator $-\Delta+V(x)$ on Euclidean space $\RR^{n}$ for a potential $V(x)$ that is periodic [BSh], almost periodic [Sh],[PS], or random ([KM] and references therein). The most common definition is to take
 
 \begin{equation}\label{eq:integrateddensityofstates}
 N(\lambda)=\lim_{k\to\infty}\frac{1}{|\Omega_{k}|}N_{k}(\lambda),
 \end{equation}
 where $\Omega_{k}$ is a sequence of regular domains increasing to $\RR^{n}$, and $N_{k}$ is the eigenvalue counting function for $-\Delta+V(x)$ on $\Omega_{k}$ with suitable boundary conditions (in the periodic case it is typical to take $\Omega_{k}$ to be cubes with periodic boundary conditions). The advantage of this definition is that it allows for numerical approximations, it avoids dealing with the spectral projections for operators that do not have discrete spectrum, and it implies monotonicity with respect to the operator. It is shown in [Sh] that for almost periodic potentials the definition agrees with our definition (2.5), where the average is the Bohr mean of almost periodic functions. The main point of this paper is to export this idea to many new contexts, and this is the justification for introducing the term ``spectral mass'' in place of ``integrated density of states.''
 
 In Section 2 we will elaborate on the idea of a spectral mass, establishing its elementary properties (positivity and monotonicity), discuss its relationship to the heat kernel, and present a conjecture related to the minmax characterization of eigenvalues that would allow simple estimates of spectral masses.
 
 In Section 3 we discuss examples where the Laplacian is  homogeneous, so that $K_{\lambda}(x,x)$ is constant and there is no problem defining the average. 
 
 For Euclidean space, we obtain~(\ref{eq:M}) exactly (no error term), while for the lattice $\ZZ^{n}$ we obtain~(\ref{eq:M}) as $\lambda\to0$ (the spectrum is bounded). For hyperbolic space we obtain (1.6) with error term $O(\lambda^{\frac{n-2}{2}})$, as a consequence of an integral formula for $M(\lambda)$. For the Heisenberg subLaplacian we obtain
 \begin{equation}
M(\lambda)=c'_{n}\lambda^{n+1} 
 \end{equation} for a constant $c'_{n}$ given by an infinite series ($c'_{1}=\frac{1}{32}$, $c'_{2}=\frac{1}{576\pi}$). Here the Heisenberg group has Euclidean dimension $2n+1$ but Hausdorff dimension $2n+2$, so $n+1$ is again half the Hausdorff dimension. For the Laplacian on a homogeneous tree of degree $q+1$, we obtain an integral formula for $M(\lambda)$ and an asymptotic formula
 \begin{equation}
M(\lambda_{0}+\epsilon) =c''_{q}\epsilon^{3/2}+O(\epsilon^{3/2})\text{ as }\epsilon\to 0^{+}
 \end{equation}where $\lambda_{0}=q+1-2\sqrt{q}$ is the bottom of the spectrum (note that the exponent $\frac{3}{2}$ is independent of $q$).
 
 These results are all simple consequences of well-known formulas for harmonic analysis on these spaces, but our results place all these disparate facts together in a coherent context. 
 
 In Section 4 we discuss some inhomogeneous examples. The simplest of these are the half-line and the Euclidean or hyperbolic half-space, with either Dirichlet or Neumann boundary conditions. In these examples $K_{\lambda}(x,x)$ can be written as a sum of the kernel for the full space plus a term that decays as $x\to\infty$. The decaying term thus contributes zero toward the average value, so $M(\lambda)$ is the same for the half-space as for the full space. More interesting examples are provided by unbounded open domains in $\RR^{n}$. We are not able to obtain the desired results for such examples, but we sketch three different approaches to studying the problem. The issues here are quite technical since it is necessary to impose some hypothesis on the domain. This is an area that is certainly ripe for future development.
 
 In Section 5 we come to our second main heuristic idea, which we call the asymptotic splitting law. Suppose the space $X$ and the Laplacian have a finite symmetry group $G$. Then $G$ has a finite number of inequivalent unitary representations $\{\pi_{j}\}$ with  dimension $\{d_{j}\}$, and the dimension formula says 
 \begin{equation}
\sum_{j}d^{2}_{j}=\# G.
\end{equation} We may sort the eigenspaces into subspaces that transform according to the representations $\{\pi_{j}\}$, and write $M(\lambda)$ as the sum of $M_{j}(\lambda)$ where each $M_{j}(\lambda)$ gives the contribution to the spectral mass arising from the representation $\pi_{j}$. The asymptotic splitting law heuristic is
\begin{equation}
\lim_{\lambda\to\infty}\frac{M_{j}(\lambda)}{M(\lambda)}=\frac{d_{j}^{2}}{\# G}.
\end{equation} This statement is not always valid, as we demonstrate by simple counter-examples . For it to have any hope of being valid, we must have that most orbits under the $G$ action on $X$ be full orbits. But even more than this, we need the existence of a fundamental domain for the action that has boundary that is relatively small. We give a precise estimate for Laplacians on finite graphs for the difference between $M_{j}(\lambda)/M(\lambda)$ and $d_{j}^{2}/\# G$ for all values of $\lambda$. In some noncompact settings we can show that these quantities are exactly the same for all $\lambda$. One interesting example we discuss is the Sierpinski gasket and related fractals. Here we are able to establish (1.11) with a much smaller error term than we would expect for nonfractal examples.

In Section 6 we discuss similar questions when the symmetry group $G$ is a compact Lie group. In this case there is no dimension formula, but we can still sort the eigenspaces according to the irreducible representations, and we can compare $M_{j}(\lambda)/M_{k}(\lambda)$ with $d^{2}_{j}/d^{2}_{k}$ for different representations. 

In section 7 we discuss a noncompact fractal example, the generic infinite blowup of the Sierpinski gasket. The structure of the spectrum of the Laplacian here has been described by Teplyaev [Te]; there is an orthonormal basis of compactly supported eigenfunctions, but each eigenspace has infinite multiplicity, and   the closure of the set of eigenvalues is a Cantor set. We present evidence that the spectral mass function exists and has the same asymptotics  as the eigenvalue counting function on the Sierpinski gasket. 

Finally, in Section 8, we discuss the relationship between the spectral mass function on a space $X$ and a covering space $\widetilde{X}$ of $X$. In typical examples $X$ is compact and $\widetilde{X}$ noncompact, and we want to pass from information about the spectral mass function $\widetilde{M}(\lambda)$ on $\widetilde{X}$ to information about $M(\lambda)$ (and hence $N(\lambda)$) on $X$ from the identity
\begin{equation}
K_{\lambda}(x,x)=\sum_{\gamma\in\Gamma}\widetilde{K}_{\lambda}(\gamma \widetilde{x},\widetilde{x}),
\end{equation} where $\widetilde{x}$ projects to $x$ and $\Gamma$ is the group of deck transformations. The term corresponding to $\gamma=e$, the identity, contributes to $M(\lambda)$ exactly the value $\widetilde{M}(\lambda)$. We could conclude that $M(\lambda)$ has the same asymptotics as $\widetilde{M}(\lambda)$ if we could show that the sum of all the other terms on the right side of (1.12) has slower growth rate. We give some simple examples where this is the case, but with the caveat that the infinite series is only conditionally convergent. We discuss the example of hyperbolic manifolds of finite volume, where such conclusions would be very interesting, but the technical challenge of estimating (1.12) seems difficult. For the example of the compact quotients of the Heisenberg group, we note that the desired estimates have been obtained by Taylor [T] using heat kernel methods. Such an approach might well be more effective for dealing with other examples discussed in this section. For a direct approach to these quotients, see Folland [Fo].

There are many earlier words that are closely related to our ideas. H\"ormander [H] gives estimates for the asymptotics of $K_{\lambda}(x,x)$ in the case of an elliptic pseudodifferential operator on a manifold, but the estimates are not uniform in $x$ at infinity or near the boundary (but also see [I]), and so do not immediately yield information about $M(\lambda)$. That paper also gives references to earlier work on the topic. The paper of Cheeger, Gromov and Taylor [CGT] gives estimates for the kernels of many functions of the Laplacian based on wave equations methods.  Wave equation methods were introduced by Levitan [L] in 1952 (see also [SV1],[SV2]). Many of the expressions for $K_{\lambda}$ in the examples in section 3 may be found in our earlier paper [S1] and in the book of Taylor [T].

This paper is woefully incomplete, and is offered in the spirit of adventure. We hope that adventurous readers will be inspired to tackle some of the conjectures and questions that are scattered throughout the text.
 
 \section{Spectral mass}

Let $E_{\lambda}$ denote the spectral resolution of a nonnegative self-adjoint operator $A$ on $L^{2}(X,\mu)$, so $E_{\lambda}$ is the projection operator onto the $[0,\lambda]$ portion of the spectrum of $A$. In many cases $E_{\lambda}$ is an integral operator with a continuous kernel $K_{\lambda}(x,y)$:
\begin{equation}
E_{\lambda}f(x)=\int_{X}K_{\lambda}(x,y)f(y)d\mu(y).
\end{equation}

For example, if $A$ has a discrete spectrum of $L^{2}$ eigenfunctions $\varphi_{j}$ with eigenvalues $\lambda_{j}$, 
\begin{equation}
A\varphi_{j}=\lambda_{j}\varphi_{j},
\end{equation} $\{\varphi_{j}\}$ an orthonormal basis of $L^{2}$, then if $E_{\lambda}$ has finite dimensional range we have
\begin{equation}
K_{\lambda}(x,y)=\sum_{\lambda_{j}\leq \lambda}\varphi_{j}(x)\overline{\varphi_{j}(y)}\;\;\;\;\text{  (finite sum).}
\end{equation}

In this case,
\begin{equation}
\int K_{\lambda}(x,x)d\mu(x)=N(y)=\#\{\lambda_{j}\leq\lambda\}=\dim\mbox{Range }E_{\lambda}.
\end{equation}

More generally, we would like to define a \emph{spectral mass function} $M(\lambda)$ that measures the ``size'' of projection $E_{\lambda}$. We propose taking
\begin{equation}
M(\lambda)=\mbox{Average}(K_{\lambda}(x,x)).
\end{equation}

Of course this is not a precise definition, since we have to say what we mean by ``Average'', and this will have to depend on the context. Typically the average will be defined by
\begin{equation}
\lim_{k\to\infty}\frac{1}{\mu(B_{k})}\int_{B_{k}}K_{\lambda}(x,x)d\mu(x)
\end{equation}
for some ``reasonable'' sequence of sets $B_{k}$ increasing to $X$, where the limit is independent of the choice of the reasonable sequence. We may have to settle for $\liminf$ and $\limsup$ in place of $\lim$ in (2.6). Two situations where the meaning of Average is obvious:

(i) if $\mu(x)<\infty$ then the average is just
\begin{equation}
\frac{1}{\mu(X)}\int_{X}K_{\lambda}(x,x)d\mu(x), \;\;\;\text{ in which case}
\end{equation}
\begin{equation}
M(\lambda)=\frac{1}{\mu(X)}N(\lambda); \text{ and}
\end{equation}

(ii) if $K_{\lambda}(x,x)$ is constant in $X$, then the Average is this constant.

If $X$ is a metric space we would hope to choose $B_{k}$ to be the balls of radius $k$ centered at some point $z$, with (2.6) independent of the choice of $z$.

The following known result provides some basic properties of $K_{\lambda}(x,x)$. We provide the simple proofs for the benefit of the reader.

\begin{lem}
a) $K_{\lambda}(x,x)$ is nonnegative (possibly $+\infty$).

b) $K_{\lambda}(x,x)$ is monotone increasing in $\lambda$.
\end{lem}

\begin{proof}
a) Since $E_{\lambda}$ is a projection
\begin{equation}
K_{\lambda}(x,z)= \int_{X}K_{\lambda}(x,y)K_{\lambda}(y,z)d\mu(y)
\end{equation} If we set $x=z$ and observe that $K_{\lambda}(y,x)=\overline{K_{\lambda}(x,y)}$ because $K_{\lambda}$ is self-adjoint, (2.9) yields
\begin{equation}
K_{\lambda}(x,x)=\int_{X}|K_{\lambda}(x,y)|^{2}d\mu(y).
\end{equation}

b)  Suppose $\lambda'<\lambda$. Because $E_{\lambda}$ is a spectral resolution, $E_{\lambda}-E_{\lambda'}$ and $E_{\lambda'}$ are orthogonal projections ($E_{\lambda}-E_{\lambda'}$ is the spectral projection onto the $[\lambda',\lambda]$ portion of the spectrum) with kernels $K_{\lambda}-K_{\lambda'}$ and $K_{\lambda'}$. The orthogonality implies
\begin{equation}
\int_{X}(K_{\lambda}(x,y)-K_{\lambda'}(x,y))K_{\lambda'}(y,z)d\mu(y)=0,
\end{equation} so the same reasoning that leads to (2.10) shows
\begin{equation}
K_{\lambda}(x,x)-K_{\lambda'}(x,x)=\int_{X}|K_{\lambda}(x,x)-K_{\lambda'}(x,y|^{2}d\mu(y)
\end{equation}
\end{proof}

The lemma shows that if we can make sense of (2.5), then the spectral mass will be nonnegative and monotone increasing in $\lambda$.

It is interesting to compare the spectral mass with a more familiar object, the heat kernel on the diagonal. Let $h_{t}(x,y)$ denote the kernel for the heat semigroup $e^{t\Delta}$. There is a vast literature on estimates for heat kernel, both on and off the diagonal ([G]). Typical results are upper
\begin{equation}
h_{t}(x,x)\leq c_{1}\Phi(t)
\end{equation}
and lower
\begin{equation}
h_{t}(x,x)\geq c_{1}\Phi(t)
\end{equation} on diagonal bounds, for specific (or sometimes generic) functions $\Phi(t)$, often a negative power $\Phi(t)=t^{-b}$ for some $b>0$. Define
\begin{align}
M^+(\lambda)&=\limsup_{k\to\infty}\frac{1}{\mu(B_{k})}\int_{B_{k}}K_{\lambda}(x,x)d\mu(x),\\
M^{-}(\lambda)&=\liminf_{k\to\infty}\frac{1}{\mu(B_{k})}\int_{B_{k}}K_{\lambda}(x,x)d\mu(x)\notag
\end{align} and
\begin{align}
H^+(t)&=\limsup_{k\to\infty}\frac{1}{\mu(B_{k})}\int_{B_{k}}h_{t}(x,x)d\mu(x),\\
H^{-}(t)&=\liminf_{k\to\infty}\frac{1}{\mu(B_{k})}\int_{B_{k}}h_{t}(x,x)d\mu(x).\notag
\end{align} Since we know

\begin{align}
h_{t}(x,x)=\int_{0}^{\infty}te^{-\lambda t}K_{\lambda}(x,x)d\lambda
\end{align} we may draw various conclusions relating these quantities. For example, the upper estimate (2.13) implies
\begin{align}
H^{+}(t)\leq c_{1}\Phi(t),
\end{align} and similarly (2.14) implies
\begin{align}
H^{-}(t)\geq c_{2}\Phi(t).
\end{align}

On the other hand (2.17) implies 
\begin{align}
K_{\lambda}(x,x)\leq eh_{1/\lambda}(x,x)\;\;\;\text{hence}
\end{align}
\begin{align}
\frac{1}{\mu(B_{k})}\int_{B_{k}}K_{\lambda}(x,x)d\mu(x)\leq \frac{e}{\mu(B_{k})}\int_{B_{k}}h_{1/\lambda}(x,x)d\mu(x)
\end{align} and finally
\begin{equation}
M^{\pm}(\lambda)\leq e H^{\pm}(1/\lambda)
\end{equation} So the upper estimate (2.13) implies
\begin{equation}
M^{+}(\lambda)\leq ec_{1}\Phi(1/\lambda).
\end{equation}

Of course, if we know power law limits for the heat kernel, we can use (2.17) and the Karamata Tauberian Theorem to obtain power limits for $K_{\lambda}(x,x)$.

One of the main tools is estimating the eigenvalue counting function $N(\lambda)$ is the Rayleigh quotient 
\begin{equation}
\mathcal{R}(u)=\frac{\mathcal{E}(u,u)}{||u||^{2}},
\end{equation} where $\mathcal{E}$ is the energy associated to the Laplacian and the measure via
\begin{equation}
\mathcal{E}(u,v)=-\int(\Delta u)vd\mu
\end{equation} The minmax characterization of eigenvalues leads to
\begin{equation}
N(\lambda)=\max\{\dim L:\mathcal{R}|_{L}\leq \lambda\}.
\end{equation} If $\Delta$ and $\Delta'$ are two Laplacians and 
\begin{equation}
\mathcal{R}(u)\leq\mathcal{R}'(u)\;\;\;\text{for all $u$},
\end{equation} where $\mathcal{R}'$ is the Rayleigh quotient for $\Delta'$, then (2.16) implies
\begin{equation}
N(\lambda)\geq N'(\lambda).
\end{equation}

Typically we start with one Laplacian and construct two simpler Laplacians $\Delta_{1}$ and $\Delta_{2}$ for which
\begin{equation}
\mathcal{R}_{1}(u)\leq\mathcal{R}(u)\leq\mathcal{R}_{2}(u).
\end{equation} If we can show that $\Delta_{1}$ and $\Delta_{2}$ have the same spectral asymptotics, then we may conclude that $\Delta$ also has the same spectral asymptotics,

In the general context, the Rayleigh quotient is well-defined by (2.24) and (2.25). 

\begin{con}
Suppose $\Delta$ and $\Delta'$ are two Laplacians for which (2.27) holds. Then
\begin{equation}
M^{\pm}(\lambda)\geq (M^{\pm})'(\lambda).
\end{equation}
\end{con}

This conjecture would allow us to use Reyleigh quotient estimates to obtain spectral mass estimates. It would appear that it would be easy to prove this conjecture in contexts where the integrated density of states definition~(\ref{eq:integrateddensityofstates}) is valid. It is also an easy consequence of the next conjecture.

\begin{con}
Let $L$  denote a closed subspace of $L^{2}(\mu)$, and $K_{L}(x,y)$ the kernel of the orthogonal projection onto $L$. Then
\begin{align}
M^{+}&=\max\left\{\limsup_{k\to\infty}\frac{1}{\mu(B_{k})}\int_{B_{k}}K_{L}(x,x)d\mu(x):\mathcal{R}|_{L}\leq\lambda\right\}\\
M^{-}&=\max\left\{\liminf_{k\to\infty}\frac{1}{\mu(B_{k})}\int_{B_{k}}K_{L}(x,x)d\mu(x):\mathcal{R}|_{L}\leq\lambda\right\}\notag
\end{align} 
\end{con}with the maximum attained for $L= $ range $E_{\lambda}$.

\section{Homogeneous spaces}

In this section we discuss some examples in which there is a transitive group of symmetries actiong on $X$ that preserves the operator $A$. In such cases the function $K_{\lambda}(x,x)$ is independent of $x$, so the Average in (2.5) is just this constant value. 

\textbf{Example 3.1: The Euclidean Laplacian}

Let $A=-\Delta$ on $\RR^{n}$. Then 
\begin{equation}
K_{\lambda}(x,y)=\frac{1}{(2\pi)^{n}}\int_{{|\xi|\leq\sqrt{\lambda}}}e^{i(x-y)\cdot\xi}d\xi
\end{equation}

It is well-known how to express $K_{\lambda}(x,y)$ in terms of Bessel functions, but for our purposes it suffices to note that the integrand is identically one where $x=y$, so 
\begin{equation}
K_{\lambda}(x,x)=c_{n}\lambda^{n/2}\;\;\;\;\;\;\;\;\;\;\;\;\;\text{for}
\end{equation}
\begin{equation}
c_{n}=\frac{1}{(2\pi)^{n}}b_{n}
\end{equation}  where $b_{n}$ is the volume of the unit ball in $\RR^{n}$. The constant is the same one that appears in the Weyl asymptotic law, but in this case the asymptotic law is exact for all values of $\lambda$.

\textbf{Example 3.2: Lattice Graph Laplacian}

Let $X=\ZZ^{n}$, the $n$-dimensional cubic lattice, and let $A$ be the negative of the graph Laplacian
\begin{equation}
-\Delta f(m)=\sum_{i=1}^{n}(2f(m)-f(m+e_{i})-f(m-e_{i}))
\end{equation} where $e_{i}$ denotes the standard basis elements in $\ZZ^{n}$. (Note: it is sometimes conventional to divide by $2n$ in the definition, but this only changes the computations below by a constant.) The theory of Fourier series tells us that reasonable functions $f$ may be written 
\begin{equation}
f(m)=\int_{C}\sum_{k\in\ZZ^{n}}e^{2\pi i(m-k)\cdot\xi}f(k)d\xi
\end{equation} where $C$ denotes the cube $\{\xi:|\xi_{i}|\leq\frac{1}{2}\}$ in $\RR^{n}$, and the functions $m\to e^{2\pi m\cdot\xi}$ are eigenfunctions of $-\Delta$ with eigenvalue $\sum_{i=1}^{n}2(1-\cos2\pi\xi_{i})=\sum_{i=1}^{n}4\sin^{2}\pi\xi_{i}$. This shows that the spectrum of $-\Delta$ is $[0,4n]$ and the spectral projection operator $E_{\lambda}$ is obtained from (3.5) by restricting the integral to 
\begin{equation}
C\cap\left\{\sum_{i=1}^{n}4\sin^{2}\pi\xi_{i}\leq\lambda\right\}.
\end{equation} Thus the spectral mass function is the volume of the set (3.6). When $n=1$ this gives exactly 
\begin{equation}
M(\lambda)=\frac{2}{\pi}\sin^{-1}\left(\frac{\sqrt{\lambda}}{2}\right).
\end{equation} Note that this is asymptotic to $\frac{\sqrt{\lambda}}{\pi}$ as $\lambda\to0$.
For higher values of $n$ we do not have an exact formula for $M(\lambda)$, but for small values of $\lambda$ we may approximate $\sin^{2}\pi\xi_{i}$ by $(\pi\xi_{i})^{2}$, so
\begin{equation}
M(\lambda)\sim c_{n}\lambda^{n/2}\;\;\;\;\;\text{ as }\lambda\to0
\end{equation} for the same constant (3.3). In other words, the Weyl asymptotic law holds for the bottom of the spectrum.

\textbf{Example 3.3: The Hyperbolic Laplacian}

Let $A=-\Delta$ on $H^{n}$, the hyperbolic $n$-space, with metric $d(x,y)$. It is known that the spectrum is $\left[\left(\frac{n-1}{2}\right)^{2},\infty\right)$ and we can write
\begin{equation}
K_{\lambda}(x,y)=\int_{0}^{\sqrt{\lambda-\left(\frac{n-1}{2}\right)^{2}}}\varphi_{t}(d(x,y))dt
\end{equation} for an explicit spherical function $\varphi_{t}$ ([Ta]). For our purposes it suffices to know 
\begin{align}
\varphi_{t}(0)&=\frac{w_{n-1}}{(2\pi)^{n}}\left|\frac{\Gamma\left(\left(\frac{n-1}{2}\right)^{2}+it\right)}{\Gamma(it)}\right|^{2}\\
&=\begin{cases}
{\displaystyle\frac{w_{n-1}}{(2\pi)^{n}}\prod_{j=0}^{\frac{n-3}{2}}}(t^{2}+j^{2})&\text{$n$ odd}\\
{\displaystyle\frac{w_{n-1}}{(2\pi)^{n}}\;t\tanh\pi t\;\prod_{j=0}^{\frac{n-4}{2}}\left(t^2+\left(j+\frac{1}{2}\right)^{2}\right)}&\text{$n$ even}
\end{cases}\notag
\end{align} where $w_{n-1}$ denotes the surface measure of the unit sphere (so $w_{n-1}=nB_{n}$). Thus the spectral mass is exactly 
\begin{equation}
M(\lambda)=c_{n}n{\displaystyle\int_{0}^{\sqrt{\lambda-\left(\frac{n-1}{2}\right)^{2}}}}\left|\frac{\Gamma\left(\left(\frac{n-1}{2}\right)^{2}+it\right)}{\Gamma(it)}\right|^{2}dt.
\end{equation} Asymptotically we have
\begin{align}
M(\lambda)&=c_{n}n\int_{0}^{\sqrt{\lambda}}t^{n-1}dt+O\left(\lambda^{\frac{n-2}{2}}\right)\\
&=c_{n}\lambda^{\frac{n}{2}}+O(\lambda^{\frac{n-2}{2}})\;\;\;\;\;\;\;\;\;\text{as $\lambda\to\infty$.} \notag
\end{align} 
Note that the remainder term is smaller than expected. This may also be explained in terms of more refined Weyl asymptotic laws. For example, when $n=3$ the exact formula is $c_{3}(\lambda-1)^{\frac{3}{2}}$.

\textbf{Example 3.4: The Heisenberg SubLaplacian}

Let $\heis_{n}$ be $\CC^{n}\times \RR$ endowed with the group law
\begin{equation}
(z,t)\circ(z',t')=\left(z+z',t+t'-\frac{1}{2}\imm z\cdot \overline{z'}\right)
\end{equation} with left invariant vector fields
\begin{equation}
\begin{cases}
X_{j}={\displaystyle\frac{\partial}{\partial x_{j}}-\frac{1}{2}y_{j}\frac{\partial}{\partial t}}\\
Y_{j}={\displaystyle\frac{\partial}{\partial y_{j}}+\frac{1}{2}x_{j}\frac{\partial}{\partial t}}\\
T={\displaystyle\frac{\partial}{\partial t}}
\end{cases}
\end{equation} and let $A=-\mathfrak{L}$ for
\begin{equation}
\mathfrak{L}=\sum_{j=1}^{n}(X^{2}_{j}+Y_{j}^{2}),
\end{equation} the Heisenberg subLaplacian. The joint spectral resolution of the commuting operators $\mathfrak{L}$, $iT$ is known, and we may write
\begin{equation}
E_{\lambda}f=\sum_{\epsilon=\pm1}\;\;\sum_{k=0}^{\infty}\int_{0}^{\lambda}f*\varphi_{s,k,\epsilon}ds
\end{equation} where $*$ denotes the group convolution. Here $\varphi_{\lambda,k,\epsilon}$ is a the spherical function
\begin{equation}
\varphi_{\lambda,k,\epsilon}(z,t)=\frac{\lambda^{n}}{(2\pi)^{n}(n+2k)^{n+1}}\exp\left(\frac{-i\epsilon\lambda t}{n+2k}\right)\exp\left(\frac{-\lambda|z|^2}{4(n+2k)}\right)L_{k}^{n-1}\left(\frac{\lambda|z|^{2}}{2(n+2k)}\right)
\end{equation} and $L_{k}^{n-1}$ are the Laguerre polynomials characterized by the generating function identity
\begin{equation}
\sum_{k=0}^{\infty}r^{k}L_{k}^{n-1}(s)=(1-r)^{-n}e^{-\frac{rs}{1-r}}\;\;\;\;\text{for } |r|<1.
\end{equation}Thus the spectral mass is
\begin{align}
M(\lambda)&=\sum_{\epsilon=\pm1}\;\;\sum_{k=0}^{\infty}\int_{0}^{\lambda}\varphi_{s,k,\epsilon}(0,0)ds\\
&=2\sum_{k=0}^{\infty}\int_{0}^{\lambda}\frac{s^{n}}{(2\pi)^{n+1}(n+2k)^{n+1}}L_{k}^{n-1}(0)ds.\notag
\end{align} Note that (3.18) implies 
\begin{equation}
L_{k}^{n-1}(0)=\frac{1}{k!}\left(\frac{d}{dr}\right)^{k}(1-r)^{-n}|_{r=0}=\binom{n+k-1}{k}
\end{equation} and so
\begin{equation}
M(\lambda)=\left(\frac{2}{(n+1)(2\pi)^{n+1}}\sum_{k=0}^{\infty}\frac{1}{(n+2k)^{n+1}}\binom{n+k-1}{k}\right)\lambda^{n+1}
\end{equation} It is clear that the infinite series converges since the terms are $O(k^{-2})$. It appears likely that one can evaluate the constant in (3.21) in terms of the values of the zeta function for even integers. For example, when $n=1$ it is
\begin{equation}
\frac{1}{(2\pi)^{2}}\sum_{k=0}^{\infty}\frac{1}{(1+2k)^{2}}=\frac{1}{(2\pi)^{n}}\left(\zeta(2)-\frac{1}{4}\;\zeta(2)\right)=\frac{1}{(2\pi)^{2}}\frac{\pi^{2}}{8}=\frac{1}{32}
\end{equation} and for $n=2$ it is
\begin{equation}
\frac{2}{3(2\pi)^{3}}\sum_{k=0}^{\infty}\frac{k+1}{(2+2k)^{3}}=\frac{2}{3(2\pi)^{3}}\frac{1}{8}\zeta(2)=\frac{1}{576\pi}.
\end{equation}

\textbf{Example 3.5: Laplacian on Homogeneous Trees}

Let $T_{q}$ denote the homogeneous tree where each vertex has $q+1$ distinct neighbors, with $q>1$, and let $A=-\Delta$ where 
\begin{equation}
\Delta f(x)=\sum_{y\sim x}(f(y)-f(x)).
\end{equation} A priori, the spectrum of $A$ lies in $[0,2(q+1)]$, but in fact it turns out to be the smaller interval $[q+1-2\sqrt{q},q+1+2\sqrt{q}]$. The spectral resolution is due to Cartier [Ca] and may be found in [F-TN].

Define the $c$-function by
\begin{equation}
c(z)=\left(\frac{1}{q+1}\right)\frac{q^{1-z}-q^{z-1}}{q^{-z}-q^{z-1}}\;\;\;\text{for }z\in\CC
\end{equation} and the spherical functions
\begin{equation}
\phi_{z}(x)=c(z)q^{-zd(x)}+c(1-z)q^{-(1-z)d(x)}
\end{equation} where $d(x)$ denotes the graph distance to some fixed reference vertex $x_{0}\in T$. The spherical function (3.26) is an eigenfunction of $A$ with eigenvalue
\begin{equation}
\lambda=q+1-(q^{z}+q^{1-z}).
\end{equation} It turns out that the only contribution to the spectral resolution comes from $\ree z=\frac{1}{2}$, and we will write $z=\frac{1}{2}+it$ for $0\leq t\leq\frac{\pi}{\log q}$, so (3.27) becomes 
\begin{equation}
\lambda = q+1-2\sqrt{q}\cos(t\log q).
\end{equation}

The full spectral resolution may be written
\begin{equation}
f(x)=\int_{0}^{\frac{\pi}{\log q}}\phi_{\frac{1}{2}+it}*f(x)dm(t)
\end{equation} where the convolution is
\begin{equation}
\phi_{\frac{1}{2}+it}*f(x)=\sum_{y}\phi_{\frac{1}{2}+it}(d(x,y))f(y)
\end{equation} and
\begin{equation}
dm(t)=\frac{q\log q}{2\pi(q+1)}\left|c\left(\frac{1}{2}+it\right)\right|^{2}dt=\frac{4q(q+1)\log q\sin^{2}(t\log q)}{2\pi\left((q-1)^{2}+4q\sin^{2}\left(t\log q\right)\right)}\;dt.
\end{equation} The spectral projection $E_{\lambda}$ is obtained by restricting the integral on the right side of (3.29) to the values of $t$ with $\lambda\leq t$, in other words
\begin{equation}
0\leq t\leq \frac{1}{\log q}\cos^{-1}\left(\frac{q+1-\lambda}{2\sqrt{q}}\right).
\end{equation} The definition of the $c$-function was chosen to have $c(z)+c(z+1)=1$ so that $\varphi_{z}(0)=1$, so
\begin{equation}
K_{\lambda}(x,x)=\int_{I_{\lambda}}dm(t)
\end{equation} where $I_{\lambda}$ is the interval (3.32), and this is the spectral mass function. Note that $M(\lambda)=0$ for $\lambda\leq q+1-2\sqrt{q}$ (below the spectral gap), and just above the spectral gap $\lambda=q+1-2\sqrt{q}+\epsilon$ for small $\epsilon$,
\begin{align}
M(q+1-2\sqrt{q}+\epsilon)&\approx\int_{0}^{\frac{\epsilon^{1/2}}{q^{1/4}}}\frac{2q(q+1)(\log q)^{3}}{\pi(q-1)^{2}}\;t^{2}dt\\
&=\frac{2q^{1/4}(q+1)(\log q)^{3}}{3\pi(q-1)^{2}}\;\epsilon^{3/2}+O(\epsilon^{5/2}).
\end{align}

\section{Inhomogeneous spaces}

In this section we discuss examples of spaces that are not homogeneous, forcing us to come to grips with the notion of ``Average'' in (2.5).

\textbf{Example 4.1: The half-line} 

Let $X$ be the half-line $[0,\infty]$ and $A$ be $-\frac{d^{2}}{dx}$ with either Neumann or Dirichlet boundary conditions at the origin. Then $E_{\lambda}$ corresponds to $E_{\lambda}$ on the full line with the function extended by even or odd reflections. In other words,
\begin{equation}
K_{\lambda}(x,x) =\widetilde{K}_{\lambda}(x,x)\pm \widetilde{K}_{\lambda}(x,-x)
\end{equation} where $\widetilde{K}_{\lambda}$ is the kernel for the full line (Example 3.1 with $n=1$). Explicitly,
\begin{equation}
\widetilde{K}_{\lambda}(x,y)=\frac{1}{\pi}\frac{\sin\sqrt{\lambda}(x-y)}{x-y}.
\end{equation} The obvious definition of ``Average'' here is 
\begin{equation}
M(\lambda)=\lim_{R\to\infty}\frac{1}{R}\int_{0}^{R}K_{\lambda}(x,x)dx.
\end{equation} Since $\widetilde{K}_{\lambda}(x,-x)=\frac{1}{2\pi}\frac{\sin\sqrt{\lambda}x}{x}$ tends to zero as $x\to\infty$, it clearly does not contribute to (4.3) and so we obtain the same result $c_{1}\sqrt{\lambda}$ as for the full line.

\textbf{Example 4.2: The half-space}

Let $X=\RR^{n}\times[0,\infty)$ and $A$ be $-\Delta$ with either Neumann or Dirichlet conditions on the boundary. If we write $x'=(x_{1},\ldots,x_{n})\in\RR^{n}$ and $0\leq x_{n+1}<\infty$ then the analog of (4.1) is
\begin{align}
K_{\lambda}((x',x_{n+1}),(x',x_{n+1}))=\widetilde{K}_{\lambda}((x',x_{n+1}),(x',x_{n+1}))\pm\widetilde{K}_{\lambda}((x',x_{n+1}),(x',-x_{n+1}))
\end{align} where $\widetilde{K}_{\lambda}$ denotes the kernel for $\RR^{n+1}$. It is well-known that $\widetilde{K}_{\lambda}$ may be expressed in terms of Bessel functions and $\widetilde{K}_{\lambda}((x',x_{n+1}),(x',-x_{n+1}))$ tends to zero as $x_{n+1}\to\infty$. So if we define ``Average'' in the expected way,
\begin{equation}
M(\lambda)=\lim_{R\to\infty}\frac{2}{B_{n+1}R^{n+1}}\int_{|x'|^{2}+|x_{n+1}|^{2}\leq R^{2}}{K}_{\lambda}((x',x_{n+1}),(x',x_{n+1}))dx'dx_{n+1}
\end{equation} then $\widetilde{K}_{\lambda}((x',x_{n+1}),(x',-x_{n+1}))$ in (4.4) does not contribute to the limit in (4.5), so we get the same answer $c_{n+1}\lambda^{\frac{n+1}{2}}$ as in the case of $\RR^{n+1}$. There are many other ways we could define the ``Average'' that would lead to the same answer. On the other hand one could imagine some very perverse definitions that would take averages over regions that eventually exhaust the half-space but unduly weight points close to the boundary, and lead to different answers (or even nonexistence of limits).

\textbf{Example 4.3: Disjoint Cubes}

Let $\{C_{j}\}$ be an infinite sequence of open cubes in $\RR^{n}$ whose closures are disjoint, let $X=\bigcup_{j}C_{j}$ and let $A$ be the restriction of $-\Delta$ to each cube with either Neumann or Dirichlet boundary conditions on the boundary of each cube. This is not usually considered an interesting example since $X$ is not connected, but it will enable us to develop some insight that may be carried over to more interesting examples. Let $s_{j}$ denote the side length of $C_{j}$, so $\vol(C_{j})=s_{j}^{n}$.

To begin we make the following assumptions

(i) $\sum_{j}s^{n}_{j}=+\infty$ so $X$ has infinite volume, but

(ii) there exists $\epsilon>0$ such that $\sum_{s_{j}\leq\epsilon}s^{n}_{j}<\infty$, so that all sufficiently small cubes make only a finite contribution to the volume.

We define ``Average'' as the limit of averages over $\bigcup_{j=1}^{N}C_{j}$ as $N\to\infty$, but in fact this limit may not exist, so we define upper and lower averages
\begin{equation}
M^{+}=\limsup_{N\to\infty}\frac{1}{\sum_{j=1}^{N}s_{j}^{n}}\sum_{j=1}^{N}\int_{C_{j}}K_{\lambda}(x,x)dx
\end{equation} and $M^{-}(\lambda)$ with $\limsup$ replaced by $\liminf$. 
Now the projection operators $E_{\lambda}$ are given on each $C_{j}$ as just the standard Neumann or Dirichlet Laplacian projection operators on $C_{j}$, so in fact
\begin{equation}
\int_{C_{j}}K_{\lambda}(x,x)dx=N_{j}(\lambda)
\end{equation} where $N_{j}(\lambda)$ is the eigenvalue counting function on $C_{j}$.

We know that 
\begin{equation}
N_{j}(\lambda)=c_{n}s_{j}^{n}\lambda^{n/2}+\epsilon_{j}(\lambda)s_{j}^{n}\lambda^{n/2}
\end{equation} where the error term $\epsilon_{j}(\lambda)$ tends to zero as $\lambda\to\infty$. However, the error estimate is not uniform across all cubes. The two observations we need are that in any cases the error term $\epsilon_{j}(\lambda)$ is uniformly bounded for all cubes, and the vanishing $\epsilon_{j}(\lambda)$ as $\lambda\to\infty$ is uniform on all cubes with $s_{j}\geq\epsilon$ (for example, for any fixed $\lambda$, if $s_{j}$ is small enough then $N_{j}(\lambda)=1$ (Neumann) or 0 (Dirichlet)). The first observation together with assumptions (i) and (ii) mean that the small cubes with $s_{j}\leq\epsilon$ do not contribute to the limit defining $M^{\pm}(\lambda)$. The second observation then yields estimates
\begin{equation}
c_{n}\lambda^{n/2}-\epsilon^{-}(\lambda)\lambda^{n/2}\leq M^{-}(\lambda)\leq M^{+}(\lambda)\leq c_{n}\lambda^{n/2}+\epsilon^{+}(\lambda)\lambda^{n/2}
\end{equation} where $\epsilon^{\pm}\to0$ as $\lambda\to\infty$. Thus (4.9) is our version of Weyl asymptotics in this case. (Better estimates are possible, but we will not discuss them here.)

On the other hand, if we violate assumption (ii) then we can end up with essentially meaningless results. Suppose, for example, that (i) holds but $\lim_{j\to\infty}s_{j}=0$. Then for a fixed $\lambda$ we will have all but a finite number of cubes too small to support nonconstant eigenfunctions with eigenvalues $\leq\lambda$. Thus $N_{j}(\lambda)=0$ for Dirichlet boundary conditions for all but a finite number of $j$'s, so 
\begin{equation}
M^{+}(\lambda)\leq\limsup_{N\to\infty}\frac{c}{\sum_{j=1}^{N}s_{j}^{n}}=0.
\end{equation} For Neumann boundary conditions with $N_{j}(\lambda)=1$ for all but a finite number of $j$'s we have
\begin{equation}
M^{-}(\lambda)\geq\liminf_{N\to\infty}\frac{N}{\sum_{j=1}^{N}s_{j}^{n}}=+\infty.
\end{equation}

\textbf{Example 4.4: Open subsets of $\RR^{n}$}

Let $X$ be an unbounded open subset of $\RR^{n}$ with smooth boundary, and consider $-\Delta$ with either Dirichlet or Neumann boundary conditions. Under some additional assumptions we expect to obtain the estimate (4.9). This is weaker than the result for the half-space in Example 4.2, so again we could hope for the stronger conclusion under more stringent hypotheses. This is likely to be a long term project, so we restrict ourself here to outlining three prospective approaches to the problem:

(a) If Conjecture 2.3 is valid, then we can try to find unions of cubes bracketing $X$, namely 
\begin{equation}
\bigcup_{j}C_{j}\subseteq X\subseteq\bigcup_{j}C'_{j}
\end{equation} and use the estimate (4.9) for Example 4.3. For the inner approximation we can arrange for the closures of $C_{j}$ to be disjoint, but not for outer approximation, so this would require a minor adjustment of the argument (but this is no different from the case of bounded domains). The main drawback  to this approach is verifying condition (ii). For example, if $X$ is the region above the graph of a smooth function, then we would need to assume that the function rapidly approaches a constant at infinity to guarantee that small cubes have finite total volume. It is likely that a more careful argument could allow a weakening of condition (ii) to allow a more reasonable class of domains. Since Conjecture 2.3 is needed to pass from elementary Rayleigh quotient estimates that follow from (4.12) to estimate for spectral mass, it is not worth working out the details of this approach until the status of the conjecture is resolved.

(b) Heat kernel estimates of the form (2.13) and (2.14) are known to hold, so we can use these to obtain estimates (4.9). The only technical problem is to find the hypothesis on $X$ that will make the constants in (2.13) and (2.14) uniform across the unbounded set $X$. This approach is too cride to yield the more precise asymptotics that we found for the halfspace.

(c) Wave equation methods, as described in [CGT], [L], [SV1], [SV2] can be used to compare functions of the Laplacian on the ambient Euclidean space. There are two technical issues here. The first is that we cannot use the spectral projection operator directly because the sharp cutoff at frequency $\lambda$ is implemented by $\psi(-\Delta)$, where $\psi$ is a discontinuous function whose Fourier transform decays too slowly at infinity. So we have to bracket $E_{\lambda}$ by 
\begin{equation}
\psi_{1}(-\Delta)\leq E_{\lambda}\leq \psi_{2}(-\Delta)
\end{equation} where $\psi_{1}$ and $\psi_{2}$ are soft cutoff functions in $\mathcal{D}$. We may take $\psi_{j}$ to be even functions so
\begin{equation}
\psi_{j}(s)=\int_{0}^{\infty}\widehat{\psi}(t)\cos (st)dt\;\;\;\;\;\text{ and}
\end{equation} 
\begin{equation}
\psi_{j}(-\Delta)=\int_{0}^{\infty}\widehat{\psi}_{j}(t)\cos(t\sqrt{-\Delta})dt
\end{equation} represents $\psi_{j}(-\Delta)$ as an integral of wave propagator operators $\cos(t\sqrt{-\Delta})$. Because of the finite propagation speed of $\cos(t\sqrt{-\Delta})$ we have the equality of the kernel $\cos(t\sqrt{-\Delta})$ for the Laplacians $\Delta$ on $X$ and $\widetilde{\Delta}$ on $\RR^{n}$ as long as we stay a distance of $t$ away from the boundary. Let us denote by $K_{\lambda}^{(j)}$ the kernel of the operator $\psi_{j}(-\Delta)$, and by $\widetilde{K}_{\lambda}^{(j)}$ the kernel of the operator $\psi_{j}(-\widetilde{\Delta})$, and by $W_{t}$ and $\widehat{W}_{t}$ the kernels of $\cos(t\sqrt{-\Delta})$ and $\cos (t\sqrt{-\widetilde{\Delta}})$. Then (4.13) gives 
\begin{equation}
K_{\lambda}^{(1)}(x,x)\leq K_{\lambda}(x,x)\leq K_{\lambda}^{(2)}(x,x)
\end{equation} while (4.15) gives
\begin{align}
K_{\lambda}^{(j)}(x,x)&=\left.\int_{0}^{\infty}\widehat{\psi}_{j}(t)W_{t}(x,y)dt\right|_{y=x}\\
\widetilde{K}_{\lambda}^{(j)}(x,x)&=\left.\int_{0}^{\infty}\widehat{\psi}_{j}(t)\widetilde{W}_{t}(x,y)dt\right|_{y=x}
\end{align}(since the kernels $W_{t}(x,y)$ are singular at $y=x$ we must first integrate before setting $y=x$). The finite propagation speed implies
\begin{equation}
K_{\lambda}^{(j)}(x,x)-\widetilde{K}_{\lambda}^{(j)}(x,x)=\left.\int_{d(x)}^{\infty}\widehat{\psi}_{j}(t)(W_{t}(x,y)-\widetilde{W}_{t}(x,y))dt\right|_{y=x},
\end{equation} where $d(x)$ denotes the distance to the boundary. Since $\widehat{\psi}_{j}$ vanishes rapidly at infinity, it should be possible to find a favorable estimate for the right side of (4.19), under suitable assumptions on $X$, using the method of [CGT]. This remains a technical issue. Since $\widetilde{K}_{\lambda}^{(j)}(x,x)$ is independent of $x$ and has the desired $O(\lambda^{n/2})$ asymptotics, (4.16) would then yield (4.9). Once again there is no hope of getting the existence of $M(\lambda)$ form the estimates (4.16).

\textbf{Example 4.15: The hyperbolic half-space}

Consider half of $H^{n}$ (Example 3.3) with either Dirichlet or Neumann boundary conditions (in the Poincar\'e ball model of $H^{n}$ this would be a half ball). By the same reasoning as in Example 4.2 we will have exactly the same asymptotics, in this case (3.12) as for the full hyperbolic space.

More generally we could consider other unbounded domains in $H^{n}$ using methods (b) and (c) sketched in Example 4.4. Even more generally we could look at Riemannian manifolds with positive injectivity radius and positive curvature bounds. It is not clear if we would require two-sided or only one-sided curvature bounds, and whether Ricci curvature bounds would suffice.

\section{Finite symmetry groups}

Suppose the space $X$  has a finite group $G$ of symmetries ($x\to gx$ for each $g\in G$) that preserves the measure $\mu$ and the operator $A$ ($Af(gx)=A(f_{g})(x)$ for $f_{g}(x)=f(gx)$). Let $\{\pi_{j}\}$ be a complete set of inequivalent irreducible representations of dimensions ${d_{j}}$. We know the dimension formula says $\sum_{j}d_{j}^{2}=\# G$. We may also split the span $L^{2}(X)$ into an orthogonal direct sum $\oplus_{j}L^{2}_{j}(X)$ where $L^{2}_{j}$ consists of functions that transform according to the representation $\pi_{j}$. In other words, $f\in L^{2}_{j}$ if $\{f_{g}\}_{g\in G}$ spans a space on which the action of $G$ is equivalent to $\pi_{j}$. We may then define spectral mass functions $M_{j}(\lambda)$ by restricting the spectral projections $E_{\lambda}$ to the spaces $L_{j}(X)$. Clearly
\begin{equation}
M(\lambda)=\sum_{j}M_{j}(\lambda),
\end{equation} and a natural question that arises is what can be said about the ratios $M_{j}(\lambda)/M(\lambda)$. A simple answer suggests itself:

\textbf{Asymptotic splitting law heuristic}
\begin{equation}
\frac{M_{j}(\lambda)}{M(\lambda)}\to\frac{d_{j}^{2}}{\# G}\;\;\;\;\;\text{ as }\lambda\to\infty.
\end{equation} 

This heuristic is certainly not universally true. An obvious necessary condition is that most orbits under the action of $G$ be full orbits (have the same  cardinality as $G$) since orbits that are not full will not support functions that transform according to all irreducible representations of $G$ (in other words, some of the spaces $L_{j}^{2}$ will be too small). But even this condition is not sufficient, as the following simple example shows. Suppose $X$ is the 3-regular graph consisting of two copies of $\ZZ_{N}$, $\{x_{0},x_{1},\ldots, x_{N-1},x_{N}=x_{0}\}$ and $\{y_{0},y_{1},\ldots,y_{N-1},y_{N}=y_{0}\}$ with $x_{j}\sim x_{j-1}$, $x_{j}\sim x_{j+1}$, $y_{j}\sim y_{j-1}$, $y_{j}\sim y_{j+1}$ and $x_{j}\sim y_{j}$. Let $G$ be the two element group consisting of the identity and the permutation $x_{j}\longleftrightarrow y_{j}$ for all $j$. Clearly the two spaces $L_{0}^{2}(X)$ and $L_{1}^{2}(X)$ consist of the even functions $f(x_{j})=f(y_{j})$ and the odd functions $f(x_{j})=-f(y_{0})$. If $A$ is the standard negative Laplacian 
\begin{equation}
\begin{cases}
Af(x_{j})=	3f(x_{j})-f(x_{j-1})-f(x_{j+1}) -f(y_{j})\\
Af(y_{j})=3f(y_{j})-f(y_{j-1})-f(y_{j+1}) -f(x_{j})
\end{cases}
\end{equation} then it is easy to see that 
\begin{equation}
\begin{cases}
f(x_{j})=e^{2\pi ijk/N}\\
f(y_{j})=\pm e^{2\pi ijk/N}
\end{cases}
\end{equation} are the even (+) and odd (-) eigenfunctions with eigenvalues
\begin{equation}
\begin{cases}
2-2\cos\frac{2\pi k}{N} &\text{even}\\
4-2\cos\frac{2\pi k}{N} &\text{odd}
\end{cases}
\end{equation} (the multiplicites are 2 unless $k=0$ or $k=\frac{N}{2}$). Note that the even spectrum lies in $[0,4]$ and the odd spectru, lies in $[2,6]$. In other words, the bottom third of the spectrum consists entirely of even spectrum, the top third consists entirely of odd spectrum, and only in the middle have do the two interweave. Although this is a finite example, so the limit in (5.12) is meaningless, it clearly violates the spirit of the heuristic and may be used to construct actual counterexamples.

It is easy to see why this example misbehaves, and then to formulate conditions that rule out this sort of problem. If we split $X$ into two natural fundamental domains for the action of $G$, namely the two copies of $\ZZ_{N}$, there are too many edges connecting them. What we want instead is a subdivision into fundamental domains with relatively few edges connecting them. We consider first the case of Laplacians on finite graphs.

We say that $F\subset X$ is a fundamental domain for the $G$-action on $X$ if
\begin{equation}
X=\bigcup_{g\in G}gF
\end{equation} We do not insist that the images $gF$ be disjoint, but of course we want the overlaps to be small. We define the \emph{boundary} $\partial F$ to consist of all points in $F$ that either lie in, or have an edge connecting to some $gF$ with $g$ not the identity.

Note that $\#\partial P$ may vary depending on the choice of the fundamental domain, we would like to make a choice that gives close to the smallest possible value.

\begin{lem} Let $A$ be any $G$-invariant Laplacian on $L^{2}(X,\mu)$ for a finite graph that is nonnegative and self-adjoint, specifically
\begin{equation}
Au(x)=\frac{1}{\mu(x)}\sum_{y\sim x}c_{xy}(u(x)-u(y))
\end{equation} for some positive conductances $c_{xy}$. Then
\begin{equation}
\left|M_{j}(\lambda)-\frac{d^{2}_{j}}{\# G} M(\lambda)\right|\leq \frac{d^{2}_{j}\# \partial F}{\mu(X)}
\end{equation}
\end{lem}
\begin{proof}
We may work with $N(\lambda)$ and $N_{j}(\lambda)$ (defined in the obvious way) since $M(\lambda)=N(\lambda)/\mu(X)$, etc. Form the minmax characterization of eigenvalues we have
\begin{equation}
N(\lambda) =\max\{\dim L:\mathcal{R}|_{L}\leq \lambda\}
\end{equation} where 
\begin{equation}
\mathcal{R}(u)=\frac{{\displaystyle\sum_{x\sim y}c_{xy}|u(x)-u(y)|^{2}}}{{\displaystyle\sum_{x}|u(x)|^{2}\mu(x)}}
\end{equation} denotes the Rayleigh quotient and $L$ denotes a subspace of $L^{2}(X)$. Similarly
\begin{equation}
N_{\lambda}(y)=\max\{\dim L:L\subset L^{2}_{j}(X) \text{ and }\mathcal{R}|_{L}\leq\lambda\}.
\end{equation} Let $\mathcal{D}_{0}$ denote the Dirichlet domain on $F$, namely 
\begin{equation}
\mathcal{D}_{0}=\{u\in L^{2}(F):u|_{\partial F}=0\}
\end{equation} and let
\begin{equation}
N^{(0)}(\lambda)=\max\{\dim L: L\subset\mathcal{D}_{0}\text{ and }\mathcal{R}|_{L}\leq\lambda\}.
\end{equation} The key observation is that for each $u\in\mathcal{D}_{0}$ and each representation $\pi_{j}$, there are $d_{j}^{2}$ linearly independent extensions of $u$ to $X$ and the Rayleigh quotient of each extension is the same as $\mathcal{R}(u)$. This immediately yields the estimates 
\begin{equation}
d_{j}^{2}N^{(0)}(\lambda)\leq N_{j}(\lambda)\leq d_{j}^{2}(N^{(0)}(\lambda)+\#\partial F)
\end{equation} and summing over $j$ we obtain 
\begin{equation}
\# GN^{(0)}(\lambda)\leq N(\lambda)\leq \# G(N^{(0)}(\lambda)+\#\partial F).
\end{equation} Combining (5.14) and (5.15) yields 
\begin{equation}
\left|N_{j}(\lambda)-\frac{d^{2}_{j}}{\# G}\right|\leq d_{j}^{2}\#\partial F,
\end{equation} and (5.8) follows from (5.16) by dividing by $\mu(X)$.
\end{proof}

Although the lemma only deals with finite graphs, it does show that it is not necessary to go to very large values of $\lambda$ in order to get desired splitting to hold with small error; all that is needed is that $\frac{\#\partial F}{\mu(X)}$ be small in comparison to $M(\lambda)$.

\textbf{Example 5.1: Sierpinski gasket}

The Sierpinski gasket ($SG$) fractal in an equilateral triangle has the dihedral symmetry group $D_{3}$ of the triangle acting on it, and the group preserves the Kigami Laplacian ([Ki],[S2]). Since the spectrum of this Laplacian is known in great detail [GRS] it is possible to verify the asymptotic splitting law heuristic directly. However, Lemma 5.1 together with the method of spectral decimation ([FS]) gives the result with a very small error estimate. SG is the invariant set for the iterated function system (IFS) consisting if the three contraction mappings $F_{i}(x)=\frac{1}{2}x+\frac{1}{2}q_{i}$, where $\{q_{i}\}$ are the vertices of the equilateral triangle,
\begin{equation}
SG=\bigcup_{i}F_{i}(SG).
\end{equation} We construct a sequence of graphs $\Gamma_{m}$ that approximate SG as follows: $\Gamma_{0}$ is the complete graph on the vertices $\{q_{i}\}$, and inductively 
\begin{equation}
\Gamma_{m}=\bigcup_{i}F_{i}(\Gamma_{m-1}).
\end{equation} Treating $\{q_{i}\}$ as the boundary of $SG$ and each graph, we note that all the other verices of $\Gamma_{m}$ have degree 4, and so we may define a graph Laplacian
\begin{equation}
\Delta_{m}u(x)=\sum_{y\stackrel{m}{\sim}x}(u(y)-u(x))
\end{equation} at nonboundary points $x$ with either Dirichlet or Neumann boundary conditions, and this is self-adjoint for the counting measure $\mu$. The Kigami Laplacian on $SG$ is a renormalized limit of these group Laplacians,
\begin{equation}
\Delta u=\frac{3}{2}\lim_{m\to\infty} 5^{m}\Delta_{m}u
\end{equation}

The method of spectral decimation says that the spectrum of $\Delta$ is a limit of the spectra of $\Delta_{m}$ in a very precise sense: for any bottom segment $[0,\lambda]$ of  the spectrum of $\Delta$ there exists $m\approx\left(\frac{\log\lambda}{\log 5}\right)$ such that the eigenfunctions restricted to $\Gamma_{m}$ are eigenfunctions of $\Delta_{m}$ and fill out the bottom segment $[0,\lambda']$ of the spectrum of $\Delta_{m}$ (here $\lambda'\approx 5^{m}\lambda$). Since the action of the dihedral group $D_{3}$ preserves all these structures, we have $N(\lambda)=N^{(m)}(\lambda')$ and $N_{j}(\lambda)=N_{j}^{(m)}(\lambda')$ for the three representations $\pi_{0}=$ trivial representation with $d_{0}=1$, $\pi_{1}=$ alternating representation with $d_{1}=1$ and $\pi_{2}=$ 2-dimensional representation with $d_{2}=2$ (note $1^{2}+1^{2}+2^{2}=6=\# D_{3}$). So the ratio $N_{j}(\lambda)/N(\lambda)$ on $SG$ is exactly equal to the ratio $N^{(m)}_{j}(\lambda')/N^{(m)}(\lambda)$ on the graph $\Gamma_{m}$, and we may use (5.16) to estimate the difference from $d_{j}^{2}/6$. We note that the natural fundamental domain $F$ (either for $SG$ or the graphs) is just $1/6$ of the triangle bounded by the perpendicular bisectors of the triangle. Now it is easy to see that $\# \partial F=m$ for $\Gamma_{m}$ which is $O(\log \lambda)$. Thus (5.16) implies
\begin{equation}
N_{j}(\lambda)=\frac{d_{j}^{2}}{6}N(\lambda)+O(\log\lambda).
\end{equation} In this case $N(\lambda)=O(\lambda^{\log 3/\log 5)})$.

Another way to explain the small reminder in this case is that $\partial F$ on $SG$ is just a countable set, hence a set of dimension zero.

In fact it is possible to improve the error estimate in (5.21) for the 2-dimensional representation $\pi_{2}$ so that the error is zero! Out of the first $3N$ eigenfunctions, exactly $2N$ are associated to $\pi_{2}$ (this requires a proper sorting of the high multiplicity eigenspaces). Another way of saying this is that the eigenfunctions come in triples, two associated with $\pi_{2}$ and the remaining are associated to either $\pi_{0}$ or $\pi_{1}$. This can be seen by examining the explicit description of the spectrum via the method of spectral decimation [GRS], or by a Rayleigh quotient argument given in [ASST] for a related fractal, the pentagasket (spectral decimation does not apply to this fractal, but it is post-critically finite (PCF)). In fact, the reasoning in [ASST] yields the following.

\begin{prop}
For a large class of PCF fractals with dihedral $D_{k}$ symmetry group, it is possible to sort the eigenfunctions (with increasing eigenvalue order) into groups of $k$, with each group containing all $\left[\frac{k-1}{2}\right]$ 2-dimensional representations and half (one when $k$ is odd, and 2 when $k$ is even) the 1-dimensional representations. 
\end{prop}

In particular, for the 2-dimensional representations
\[N_{j}(\lambda)=\frac{2}{k}N(\lambda)\pm k,\]
while for the 1-dimensional representations we only have
\[N_{j}(\lambda)=\frac{1}{2k}N(\lambda)+O(\log\lambda).\]

Even more surprising, there is experimental evidence that for the Sierpinski carpet with $D_{4}$ symmetry, something very close to Proposition 5.2 holds: there are some ``mistakes'' in the groupings but the mistakes eventually get ``corrected''. This will be discussed in [BKS].

\textbf{Example 5.2: The $2$-torus}

The dihedral group $D_{4}$ acts on the 2-torus $\RR^{2}/\ZZ^{2}$ and is a symmetry group of the standard Laplacian. It has four 1-dimensional representations and one 2-dimensional representation ($1^{2}+1^{2}+1^{2}+1^{2}+2^{2}=8$). The eigenfunctions of the Laplacian are just $e^{2\pi i(k_{1}x_{1}+k_{2}x_{2})}$ with $(k_{1},k_{2})\in\ZZ^{2}$, and the action of $D_{4}$ on $\ZZ^{2}$ describes the action on the eigenspaces. Note that a generic point in $\ZZ^{2}$ has an eight element orbit, and the associated 8-dimensional eigenspace (it may be a subspace of a larger eigenspace) splits exactly as expected. On the other hand, a four element orbit corresponds to just the trivial representation. In other words, if we consider the fundamental domain $F=\{(k_{1},k_{2}): k_{1}\geq k_{2}\geq 0\}$ with boundary $\partial F=\{(k_{1},k_{2}):k_{1}=k_{2}>0\text{ or }k_{2}=0\}$, then 
\begin{equation}
\left|N_{j}(\lambda)-\frac{d_{j}^{2}}{8}N(\lambda)\right|\leq\# \left\{(k_{1},k_{2})\in\partial F: k_{1}^{2}+k_{2}^{2}\leq\frac{\lambda}{4\pi^{2}}\right\}=O(\lambda^{1/2}).
\end{equation}
The same estimate is valid on the square with either Dirichlet or Neumann boundary values.

It seems likely that this example is typical for Laplacians on compact manifolds.

\begin{con}
Consider the Laplace-Beltrami operator on a compact Riemannian manifold of dimension $n$ with a finite symmetry group $G$ (if the manifold has a boundary, assume the boundary is sufficiently regular, and take Dirichlet or Neumann boundary conditions). Then
\begin{equation}
N_{j}(\lambda)-\frac{d_{j}^{2}}{\# G}N(\lambda)=O(\lambda^{\frac{n-1}{2}}) \;\;\;\;\;\;\;\;\text{ as } \lambda\to\infty.
\end{equation}
\end{con}

It should be possible to prove this conjecture either by modifying a proof of the Weyl asymptotics to take into account the $G$ action, or else to realize the spaces $L^{2}_{j}(X)$ as $L^{2}$ spaces on a fundamental domain $F$ with certain boundary conditions, and then observing that the Weyl asymptotic law holds for those boundary conditions ([SV1],[SV2]).

What is perhaps more interesting is that in the noncompact setting we may be able to strengthen the statement to their exact identity
\begin{equation}
M_{j}(\lambda)=\frac{d_{j}^{2}}{\# G}M(\lambda)\;\;\;\;\;\;\;\;\;\text{ for all }\lambda.
\end{equation} We illustrate this in the following examples.

\textbf{Example 5.3: The square lattice}

We return to Example 3.2, taking $n=2$, and consider the $D_{4}$ symmetry. In this case the parameter space $C$ for the eigenfunctions is the square $\left[-\frac{1}{2},\frac{1}{2}\right]\times\left[-\frac{1}{2},\frac{1}{2}\right]$ and $D_{4}$ acts also on $C$. Almost every orbit is a full eight element orbit, and so all representations are involved in the 8-dimensional eigenspace spanned by $e^{2\pi i(\pm m_{1}\xi_{1}\pm m_{2}\xi_{2})}$ and  $e^{2\pi i(\pm m_{2}\xi_{1}\pm m_{1}\xi_{2})}$. For example, the trivial representation $\pi_{0}$ occurs in the average of these eight eigenfunctions, which is $\frac{1}{2}\cos 2\pi m_{1}\xi_{1}\cos 2\pi m_{2}\xi_{2}+\frac{1}{2}\cos2\pi m_{2}\xi_{1}\cos2\pi m_{1}\xi_{2}$. Thus $M_{0}(\lambda)$ is the average over $\ZZ^{2}$ of 
\begin{equation}
\int_{C\cap\{4\sin^{2}\pi\xi_{1}+4\sin^{2}\pi\xi_{2}\leq\lambda\}}\frac{1}{4}(\cos2\pi m_{1}\xi_{1}\cos 2\pi m_{2}\xi_{2}+\cos2\pi m_{2}\xi_{1}\cos 2\pi m_{1}\xi_{2})^{2}d\xi.
\end{equation} But the average of the integrand is $\frac{1}{8}$ for almost every $\xi$, and it is straightforward to interchange the average and the integral to obtain $M_{0}(\lambda)=\frac{1}{8}M(\lambda)$. The other 1-dimensional representations are given by similar expressions, so (5.24) holds for all of them, and we obtain the same result for the 2-dimensional representation by subtraction. 

\textbf{Example 5.4: The plane with dihedral symmetry}

We return to Example 3.1 with $n=2$ and the dihedral symmetry group $D_{N}$. For the trivial representation $\pi_{0}$
\begin{equation}
K^{(0)}_{\lambda}(x,y)=\frac{1}{2N}\sum_{\gamma\in D_{N}}K_{\lambda}(x,\gamma y)
\end{equation} where $K_{\lambda}$ is given by (3.1). The contribution to $K_{\lambda}^{(0)}(x,x)$ from $\gamma=$ identity is exactly $\frac{1}{2N}K_{\lambda}(x,x).$ For all other $\gamma$, $|x-\gamma x|$ on average is large, and $K_{\lambda}(x,\gamma x)$ tends to zero as $|x-\gamma x|\to\infty$ because of the decay of Bessel functions $J_{0}$. Thus $M_{0}(\lambda)=\frac{1}{2N}M(\lambda)$ exactly, and a similar argument establishes (5.24) for all representations.

\section{Lie group symmetries}

Suppose the group of symmetries of $G$ considered in the previous section is a compact Lie group. Then the family $\{\pi_{j}\}$ of inequivalent irreducible representations of $G$ is infinite, but the dimensions $\{d_{j}\}$ are all finite. So the definitions of $L^{2}_{j}(x)$ and $M_{j}(\lambda)$ make sense, and (5.1) holds. But there is no dimension formula, and the ratio in (5.2) is likely to tend to zero, and so gives no information. Our heuristic in this case is
\begin{equation}
\frac{M_{j}(\lambda)}{M_{k}(\lambda)}\to\frac{d^{2}_{j}}{d^{2}_{k}}\;\;\;\;\;\;\;\text{ as }\lambda\to\infty
\end{equation} with the limit being uniform if we restrict $j,k\leq N$ for any finite $N$.

For this to have any hope of being valid, we must have most orbits under $G$ to be full orbits, or equivalently, the stabilizer subgroups of most points must be trivial. For example, the $n$-sphere $S^{n}$ has symmetry group $SO(n+1)$ (or $O(n+1)$), and the eigenfunctions of the invariant Laplacian are the spaces of spherical harmonics. When $n\geq 3$ there are representations of $SO(n+1)$ that do not even appear in $L^{2}(S^{n})$, and in the case $n=2$, even though all representations appear, the ratios in (6.1) tend to $d_{j}/d_{k}$. Of course $S^{n}=SO(n+1)/SO(n)$, so the stabilizer subgroups are nontrivial.

If $X=G$ and the operator is the bi-invariant Laplacian (Casimir operator), then (6.1) is trivially true, since the eigenfunctions are just the entry functions of $\pi_{j}$, and so have dimension exactly $d_{j}^{2}/d^{2}_{k}$ once $\lambda$ is larger than the eigenvalues associated to $\pi_{j}$ and $\pi_{k}$.

We can modify this example to get a more interesting class of examples by considering $X=G\times M$ where $M$ is a compact Riemannian manifold and the operator is the sum of the Laplacians on each factor. Write $\{\lambda_{j}^{(G)}\}$ for the eigenvalues associated with the representations $\{\pi_{j}\}$ and $\{\lambda_{k}^{(M)}\}$ for the eigenvalues (repeated in case of multiplicity) of the Laplacian on $M$. Then 
\begin{align}
N_{j}(\lambda)&=(\{\#k:\lambda_{k}^{(M)}\leq\lambda-\lambda_{j}^{(G)}\})d^{2}_{j}\\
&=N^{(M)}(\lambda-\lambda_{j}^{(G)})d_{j}^{2}.\notag
\end{align}
Since $N^{(M)}$ satisfies Weyl's asymptotic law, it is clear that $N^{(M)}(\lambda-a)/N^{(M)}(\lambda-b)\to1$ as $\lambda\to\infty$, so (6.2) implies (6.1).

Similarly, we can consider $X=G\times M$ where $M$ is a space with a Laplacian for which $M^{(M)}(\lambda)$ is well defined and satisfies a power law asymptotic behavior. Then
\begin{equation}
M_{j}(\lambda)=d^{2}_{j}M^{(M)}(\lambda-\lambda_{j}^{(G)})
\end{equation} and (6.1) holds. 

We expect similar results to hold for various warperd product metrics on $G\times M$, provided the orbits under $G$ remain uniformly bounded in size. The following example shows that we may need to modify the notion of Average in some cases.

\textbf{Example 6.1:} Let $X=\RR^{2}$ with the standard action of the rotation group $SO(2)$. The one-dimensional representations $\{\pi_{j}\}$ associated to $e^{ij\theta}$ are naturally indexed by $j\in\ZZ$, and $L^{2}_{j}$ consists of functions $f(r\cos v,r\sin v)=e^{ijv}f_{0}(r)$, and we project $L^{2}(X)$ onto $L^{2}_{j}(X)$ by 
\begin{equation}
\frac{1}{2\pi}\int_{0}^{2\pi}f(r\cos(v-u),r\sin(v-u))e^{iju}du=\Pi_{j}\;f(r\cos v,r\sin v)
\end{equation} and we recover $f$ from its projection by
\begin{equation}
f(r\cos v,r\sin v)=\sum_{j=-\infty}^{\infty}\Pi_{j}\;f(r\cos v,r\sin v).
\end{equation} The kernel of $E_{\lambda}$ then splits into
\begin{equation}
\sum_{j=-\infty}^{\infty} \Pi_{j}K_{\lambda}((r\cos v,r\sin v),y)
\end{equation} and so we would like $M_{j}(\lambda)$ to be some sort of Average of $K_{\lambda}^{(j)}(x,x)$ given by
\begin{equation}
\frac{1}{2\pi}\int_{0}^{2\pi}K_{\lambda}((r\cos(v-u),r\sin(v-u))e^{iju}du.
\end{equation} By (3.1) this is
\begin{equation}
\frac{1}{(2\pi^{3})}\int_{0}^{2\pi}\int_{0}^{\sqrt\lambda}\int_{0}^{2\pi}e^{isr\cos(v-u-\theta)}e^{-isr\cos(v-\theta)}e^{iju}sd\theta ds du.
\end{equation} Using the identity 
\begin{equation}
\frac{1}{2\pi}\int_{0}^{2\pi}e^{it\cos(a-u)}e^{iju}du=i^{j}J_{j}(t)e^{ija}
\end{equation} twice, we transform (6.8) into
\begin{equation}
\frac{1}{(2\pi)^{2}}\int_{0}^{\sqrt\lambda}\int_{0}^{2\pi}i^{j}J_{j}(sr)e^{ij(v-\theta)}e^{-isr\cos(v-\theta)}sd\theta ds
\end{equation} and then into
\begin{equation}
\frac{1}{2\pi}\int_{0}^{\sqrt\lambda}J_{j}(sr)^{2}sds=\frac{1}{2\pi r^{2}}\int_{0}^{\sqrt\lambda r}J_{j}(s)^{2}s ds.
\end{equation} Since (6.10) is $O\left(\frac{1}{r}\right)$ as $r\to\infty$ we cannot use the usual average on $\RR^{2}$ to define $M_{j}(\lambda)$. However, if define 
\begin{equation}
M_{j}(\lambda)=\lim_{R\to\infty}\frac{1}{R}\int_{|x|\leq R}K_{\lambda}^{(j)}(x,x)dx
\end{equation}then the limit exists and equals $\sqrt\lambda/2\pi^{2}$ for all $j$ (see [S] section 3). Thus (6.1) holds.


\section{A fractal example}

In this section we discuss an example where the operator is a fractal Laplacian. For compact fractals, where the Laplacian has discrete spectrum, it is usually the case that the eigenvalue counting function has a power order growth, $N(\lambda)\sim \lambda^{\alpha}$, but the ratio $N(\lambda)/\lambda^{\alpha}$ usually does not have a limit [KL]. For example, for the Kigami Laplacian ([Ki], [S2]) on the Sierpinski gasket ($SG$), with either Dirichlet or Neumann boundary conditions, the power $\alpha=\log 3/\log 5$, and the ratio approaches a discontinuous multiplicative periodic function as $\lambda\to\infty$. The power is rather easy to explain, since a contraction of $1/3$ in measure increases the eigenvalue by a factor of 5. The value $2\alpha$ has been mistakenly interpreted as a ``spectral dimension'' under the erroneous assumption that the Laplacian is an operator of order 2, whereas in fact it is an operator of order $d+1$, where $d=\log 3/\log (5/3)$ is the dimension of $SG$ in the effective resistance metric [S2].

The example we look at is a generic infinite blowup $SG_{\infty}$ of $SG$ [S]. It was shown by Teplyaev [Te] that the appropriately extended Laplacian $\Delta_{\infty}$ on $SG_{\infty}$ has pure point spectrum, and each eigenvalue has infinite multiplicity. More specifically, let $\Lambda=\{\lambda_{j}\}$ be the discrete spectrum of the Neumann Laplacian on $SG$. Then $5^{m}\Lambda\subset\Lambda$ for all positive integers $m$. Let $\Lambda_{\infty}=\bigcup_{m=0}^{\infty}5^{-m}\Lambda$. Then each $\lambda\in \Lambda_{\infty}$ is an eigenvalue of $-\Delta_{\infty}$ with an infinite dimensional eigenspace $\mathcal{E}_{\lambda}$ of $L^{2}$ functions (in fact there is a basis $\mathcal{E}_{\lambda}$ consisting of compactly supported functions). The actual spectrum of $\Delta_{\infty}$ is the closure of $\Lambda_{\infty}$, and is topologically a Cantor set, but in fact the eigenspaces $\mathcal{E}_{\lambda}$ for $\lambda\in\Lambda_{\infty}$ give the spectral resolution (2.1) as
\begin{equation}
E_{\lambda}f(x)=\sum_{\lambda'\leq\lambda}\int_{X}K_{\lambda'}(x,y)f(y)d\mu(y)\;\;\;\;\;\text{ for}
\end{equation}
\begin{equation}
K_{\lambda'}(x,y)=\sum_{j}\varphi_{j}(x)\overline{\varphi_{j}(y)}
\end{equation} where $\{\varphi_{j}\}$ is an ortonormal basis of $\mathcal{E}_{\lambda'}$. 

For the Laplacian on $SG$ with Neumann (or Dirichlet) conditions, the analogous sums (7.1) and (7.2) are finite, and it is known that 
\begin{equation}
M(\lambda)\sim\lambda^{\alpha}\psi(\lambda)\;\;\;\;\;\;\;\text{ as }\lambda\to\infty
\end{equation} where $\psi$ is a discontinuous bounded function (also bounded away form zero) satisfying the multiplicative periodic condition $\psi(5\lambda)=\psi(\lambda)$.\\
To understand the spectral mass function for $SG_{\infty}$ we need to simultaneously understand the averaging in (2.5) and the behavior of the infinite series in (7.1)

The infinite blowup $SG_{\infty}$ is the increasing union of finite blowups $F^{-1}_{w_{1}}\circ\cdots\circ F^{-1}_{w_{m}}(SG)=SG_{w}$ where $w=(w_{1},\ldots,w_{m})$ is a word of length $|w|=m$ with $w_{j}\in\{0,1,2\}$ and $\{F_{0},F_{1},F_{2}\}$ is the iterated function system (IFS) defining $SG$. The blowup is \emph{generic} if the infinite word $(w_{1},w_{2},\ldots)$ is not eventually constant. The measure of $SG_{w}$ is $3^{|w|}$, and the average on $SG_{\infty}$ is
\begin{equation}
\text{Ave}(f)=\lim_{|w|\to\infty}3^{-|w|}\int_{SG}f d\mu
\end{equation} if the limit exists. If we don't know that the limit exists we may consider upper Ave$^{+}(f)$ and lower Ave$^{-}(f)$ averages by taking $\limsup$ and $\liminf$ on the right side of (7.4).

Now fix a value of $w$ with $|w|=m$ and fix an eigenvalue $\lambda'\in\Lambda_{\infty}$. The basis of $\mathcal{E}_{\lambda'}$ given in [T] is not orthonormal, but it can be broken into three parts. One part consists of functions that vanish on $SG_{w}$. Although this is an infinite set, it will give zero contribution to
\begin{equation}
3^{-|w|}\int_{SG_{w}}K_{\lambda}(x,x)d\mu(x)
\end{equation}
The second part consists of functions supported in $SG_{w}$. This is nonempty only when $\lambda'\in 5^{-m'}\Lambda_{m}$ for some $m'$, and then it is finite but may be large. Its contribution to (7.5) may be understood in terms of the $5^{m'}\lambda'$ eigenspace on $SG$. The third part consists of the functions whose support intersects both $SG_{w}$ and its complement.  We will show that there are at most six functions in this part for each $\lambda'$.

\begin{lem}
For fixed $w$ and $\lambda'$ there exits an orthonormal basis of the eigenspace $\mathcal{E}_{5^{m}\lambda'}$ consisting of functions supported in either $SG_{w}$ or its complement, together with at most six eigenfunctions $\{\widetilde{\varphi}_{j}\}$ whose support intersects both $SG_{w}$ and its complement.
\end{lem}

\begin{proof}
Consider the restrictions to the boundary points of $SG_{w}$ of the functions in the eigenspace and their normal derivatives at these points. Since there are a total of six values in all, this is at most a six dimensional space. Choose a basis $\{\widetilde{\varphi}_{j}\}$ for this space. Then any eigenfunction may be written as a linear combination of $\{\widetilde{\varphi}_{j}\}$ plus an eigenfunction that vanishes together with its normal derivatives at the boundary points of $SG_{w}$. Such an eigenfunction can be split into a sum of eigenfunctions supported in $SG_{w}$ and its complement. We apply Gram-Schmidt to each of the parts separately to obtain the desired orthonormal basis.
\end{proof}

Note that the only time that there will exist nonzero eigenfunctions supported in $SG_{w}$ will be when $5^{m}\lambda'\in\Lambda$.

The contribution to (7.5) corresponding to  $\lambda'$ is 
\begin{equation}
3^{-m}\left(N_{2}(\lambda')+\sum_{j=1}^{6}\int_{SG_{w}}|\widetilde{\varphi}_{j}(x)|^{2}d\mu(x)\right)
\end{equation} where $N_{2}(\lambda')$ denotes the number of orthonormal basis in the second part, and $\{\widetilde{\varphi}_{j}\}$ are the orthonormalizations of the functions in the third part. But $N_{2}(\lambda')$ is essentially (plus or minus 3) the multiplicity of $5^{m}\lambda'$ in $\Lambda$. Thus 
\begin{equation}
3^{-m}\sum_{\lambda'\leq\lambda}N_{2}(\lambda')\approx3^{-m}\sum_{5^{m}\lambda'\leq 5^{m}\lambda}N_{\Delta}(5^{m}\lambda')=3^{-m}M_{\Delta}(5^{m}\lambda)
\end{equation} where $N_{\Delta}$ is the eigenvalue counting function for $-\Delta$ on $SG$ and $M_{\Delta}$ is the spectral mass function. In view of (7.3) and the identity $(5^{m})^{\alpha}=3^{m}$, it follows that (7.7) tends to $\lambda^{\alpha}\psi(\lambda)$ as $m\to\infty$. 

We would like to estimate the contributions of the finite sum in (7.6) to (7.5). Now $\int_{SG_{w}}|\widetilde{\varphi}_{j}|^{2}d\mu(x)=1$, and we can estimate 
\begin{equation}
\int_{SG_{w}}|\widetilde{\varphi}_{j}|^{2}d\mu\leq 3^{m}||\widetilde{\varphi}_{j}||^{2}_{\infty}.
\end{equation} Using the standard Sobolev type estimate
\begin{equation}
||u||_{\infty}\leq c\;\mathcal{E}(u,u)^{1/2} 
\end{equation} if $u$ has a zero, we obtain
\begin{equation}
||\widetilde{\varphi}_{j}||_{\infty}\leq c(\lambda')^{1/2}.
\end{equation}Combining (7.8) and (7.10) yields the estimate 
\begin{equation}
3^{-m}\sum_{j=1}^{6}\int_{SG_{w}}|\widetilde{\varphi}_{j}(x)|^{2}d\mu(x)\leq c(\lambda')^{1/2}
\end{equation} for the finite sum in (7.6). We then need to sum (7.11) over all $\lambda'\leq\lambda$.

The spectrum $\Lambda$ of $SG$ can be broken into segments of on the order of $2^{n}$ distinct eigenvalues on the order of $5^{m}$, $n=1,2,\ldots$. So there are about $2^{n}$ eigenvalues in $\Lambda$ such that $\lambda'\approx 5^{n-m'}$, and $\lambda'\leq\lambda$ means $n-m'\leq \log\lambda/\log 5$. This leads to the estimate
\begin{align}
\sum_{\lambda\leq\lambda'}(\lambda')^{1/2}&\leq c\sum_{m'}\sum_{n\leq m'+\frac{\log\lambda}{\log5}}2^{m}\cdot(5^{n-m'})^{1/2}\\
&\leq  c\sum_{m'}(2\cdot 5^{1/2})^{m'+\frac{\log\lambda}{\log 5}}5^{-m/2} \notag\\
&=c\lambda^{\frac{1}{2}+\frac{\log 2}{\log 5}}\sum_{m'}2^{m'},\notag
\end{align} which is useless because the series diverges.

Nevertheless, it is likely that the sum of the right sides of (7.11) is much smaller than the sum of the left sides. The idea is that while we cannot expect to improve on the global estimate (7.10), we only need to bound $\widetilde{\varphi}_{j}$ on $SG_{w}$, which is only a very small portion of its support for large $m'$. If the values of $\widetilde{\varphi}_{j}$ were uniformly distributed over its support, then the left side of (7.11) would be dominated by $3^{-m}$, and in place of (7.12) we would have
\begin{equation}
\sum_{\lambda'\leq\lambda}3^{-m'}\leq c\sum_{m'}\sum_{n\leq n'+\frac{\log2}{\log5}}2^{n}\cdot 3^{-m'}\leq c\lambda^{\frac{\log2}{\log5}}\sum_{m'}\left(\frac{2}{3}\right)^{m'}.
\end{equation} Note only is this finite, but it gives a lower order of growth than the main terms $\lambda^{\alpha}\psi(\lambda)$. While it is overly optimistic to expect such uniform distribution for any one particular eigenvalue, one could reasonably expect it on average in the sum over $\lambda'$.

\begin{con}
On $SG_{\infty}$, $M(\lambda)$ is well-defined and satisfies the same asymptotics (7.3) as on $SG$.
\end{con}

Of course upper estimates for the heat kernel on $SG_{\infty}$ of the form (2.13) with $\Phi(t)=t^{-\alpha}$ are known (see [BP] and [S]), so (2.23) implies the upper bound
\begin{equation}
M^{+}(\lambda)\leq ct^{\alpha}.
\end{equation} This is weaker than the conjecture, but gives supporting evidence.

Note that the integrated density of states definition~(\ref{eq:integrateddensityofstates}) easily yields the desired estimates, so another approach to proving the conjecture would be to prove the equivalence of~(\ref{eq:integrateddensityofstates}) and the spectral mass in this case.

\section{Covering Spaces}

Suppose $\widetilde{X}$ is a covering space of $X$ with covering map $\pi$ and covering group $\Gamma$, and the operator on $\widetilde{X}$ is the lift of the operator on $X$. Can we relate the spectral mass functions $\widetilde{M}(\lambda)$ and $M(\lambda)$?

If $\widetilde{K}_{\lambda}(x,y)$ is the kernel of the spectral projection on $\widetilde{X}$, then 
\begin{equation}
K_{\lambda}(x,y)=\sum_{\gamma\in\Gamma}\widetilde{K}_{\lambda}(\gamma\pi^{-1}(x),\pi^{-1}(y))
\end{equation}is the kernel of the spectral projection on $X$. Whatever notion of average we have on $X$ lifts to $\widetilde{X}$. The summand corresponding to $\gamma=$ identity $e$ in $\Gamma$ will yield $\widetilde{M}(\lambda)$ when we average over $\pi^{-1}(x)=\pi^{-1}(y)$. So we may conclude that $M(\lambda)$ and $\widetilde{M}(\lambda)$ have the same asymptotic behavior if we can show that
\begin{equation}
\mbox{Ave }\widetilde{K}_{\lambda}(\gamma z,z)\;\;\;\;\;\;\;\mbox{ for }\gamma\neq e
\end{equation}has a slower growth rate than $\widetilde{M}(\lambda)$.

\textbf{Example 8.1}: Let $X$ be the circle of $\RR/\ZZ$ and $\widetilde{X}$ be circle $\RR/N\ZZ$ for some positive integer $N$, with $\pi(x)=x\mod1$. The eigenfunctions of $-\frac{d^{2}}{dx^{2}}$ on $X$ are $e^{2\pi ikx}$ with $k\in\ZZ$ with eigenvalues $(2\pi k)^{2}$, so 
\begin{equation}
M(\lambda)=N(\lambda)=1+2\left[\frac{\sqrt\lambda}{2\pi}\right],
\end{equation}(where $[x]$ denotes the greatest integer in $x$)

 On the other hand, the eigenfunctions of $-\frac{d^{2}}{dx^{2}}$ on $\widetilde{X}$ are $e^{2\pi ikx/N}$ with $k\in\ZZ$ with eigenvalues $(2\pi k/N)^{2}$, so
\begin{equation}
\widetilde{M}(\lambda)=\frac{1}{N}\widetilde{N}(\lambda)=\frac{1}{N}\left(1+2\left[\frac{N\sqrt\lambda}{2\pi}\right]\right).
\end{equation} It is clear from inspection that
\begin{equation}
M(\lambda)=\widetilde{M}(\lambda)+O(1)\;\;\;\;\;\;\text{ as }\lambda\to\infty
\end{equation} but we would like to deduce this from (8.1) and (8.2). Here we have

\begin{equation}
\widetilde{K}_{\lambda}(x,y)=\frac{1}{N}\sum_{|k|\leq\left[\frac{N\sqrt\lambda}{2\pi}\right]}e^{2\pi i\frac{k}{N}(x-y)}
\end{equation}so (8.1) says

\begin{equation}
K_{\lambda}(x,x)=\sum_{\gamma=0}^{N}\frac{1}{N}\sum_{|k|\leq\left[\frac{N\sqrt\lambda}{2\pi}\right]}e^{2\pi i\frac{k\gamma}{N}}.
\end{equation}Of course, the $\gamma=0$ term yields (8.4), while if $\gamma\neq0$ the sum over $k$ is bounded by $N$, so the contribution to (8.7) is uniformly bounded for all $\lambda$.

\textbf{Example 8.2:} Let $X$ be again the circle $\RR/\ZZ$ but now let $\widetilde{X}=\RR$. Then, as observed in Example 3.1,
\begin{equation}
\widetilde{K}_{\lambda}(x,y)=\frac{1}{(2\pi)}\int_{-\sqrt\lambda}^{\sqrt\lambda}e^{i(x-y)\cdot\xi}d\xi=\frac{1}{\pi}\;\frac{\sin\sqrt\lambda(x-y)}{(x-y)}
\end{equation}and
\begin{equation}
\widetilde{M}(\lambda)=\frac{\sqrt\lambda}{\pi}
\end{equation} so (8.5) continues to hold. From (8.1) we have
\begin{equation}
K_{\lambda}(x,x)=\frac{\sqrt\lambda}{\pi}+\sum_{\gamma\neq0}\frac{1}{\pi}\;\frac{\sin\sqrt\lambda\gamma}{\gamma}.
\end{equation} The infinite series in (8.10) is not absolutely convergent, but it is easy to see by summation by parts that it is uniformly bounded.

Of course (8.10) is a consequence of the Poisson summation formula, and a similar description works for $\RR^{n}$ covering $\RR^{n}/\ZZ^{n}$. The problem with nonabsolute convergence in (8.10) may be overcome by replacing the spectral projections with regularized approximations.

\textbf{Example 8.3: Compact hyperbolic manifolds }

Let $\widetilde{X}$ be the hyperbolic spaces $H^{n}$ of example 3.3, and let $\Gamma$ be a cocompact discrete group of isometries, so $X$ is a compact manifold of constant negative curvature. Then $\widetilde{K}_{\lambda}$ is given by (3.9), but now we need to have an expression for $\varphi_{t}(d(x,y))$ for all values of $d(x,y)$, not just zero. The result is known and can be given explicitly in terms of Legendre functions, but to simplify the discussion we take $n=3$, when
\begin{equation}
\varphi_{t}(\cosh\gamma)=\frac{1}{2\pi^{2}}\;\frac{t\sin tr}{\sinh r}
\end{equation} and so
\begin{align}
\widetilde{K}_{\lambda}(x,y)&=\int_{0}^{\sqrt{\lambda-1}}\varphi_{t}(d(x,y))dt\\
&=\frac{1}{2\pi^{2}}\frac{\sin\sqrt{\lambda-1}\;r-\sqrt{\lambda-1}\;r\cos\sqrt{\lambda-1}\;r}{r^{2}\sin  hr}  \notag
\end{align}for $d(x,y)=\cosh r$.

Note that in the sum
\begin{equation}
K_{\lambda}(x,x)=\sum_{\gamma\in\Gamma}\widetilde{K}_{\lambda}(\gamma x,x)
\end{equation}the choice $\gamma=e$ corresponds to $r=0$ and gives the value $c_{3}(\lambda-1)^{3/2}$ for all $x$. Note that for $\gamma\neq e$ we have $r\neq0$, and (8.12) has growth of order $O(\lambda^{1/2})$, and moreover the dependence on $r$ has exponential decrease as $r\to\infty$. The values of $d(\gamma x,x)$ will increase as $\gamma\to\infty$ in $\Gamma$, but $r=\cosh^{-1} d(\gamma x, x)$ will increase at a slower rate. The values also depend on $x$, but most likely this is not too significant. It seems plausible that 
\begin{align}
&\sum_{\gamma\neq e}\widetilde{K}_{\lambda}(\gamma x,x)=\\ &\sum_{\gamma\neq e}\frac{1}{2\pi^{2}}\;\frac{\sin\sqrt{\lambda-1}\cosh^{-1}d(\gamma x,x)-\sqrt{\lambda-1}\cosh^{-1}(d(\gamma x, x))\cos\sqrt{\lambda-1}\cosh^{-1}(d(\gamma x,x))}{(\cosh^{-1}d(\gamma x,x))^{2}\sqrt{d(\gamma x,x)^{2}-1}}\notag
\end{align} is of order $O(\lambda^{1/2})$ independent of $x$. Note that this may require further assumptions about the group $\Gamma$, and it may be that the infinite series is not absolutely convergent.

Of course we know from the Weyl asymptotic law that the leading term of the asymptotic behavior of $M(\lambda)$ is $c_{3}\lambda^{3/2} $, with remainder term $O(\lambda^{1})$. The whole point of this approach is to reduce the remainder estimate to $O(\lambda^{1/2})$.

\textbf{Example 8.4: Finite volume hyperbolic manifolds}

Consider the same setting as the previous example, but now do not assume that $\Gamma$ is cocompact, but only that $X=\widetilde{X}/\Gamma$ has finite volume. We cannot assert that the Laplacian has discrete spectrum, so we must use $M(\lambda)$ rather than $N(\lambda)$. However, since $X$ has finite volume, the average in defining $M(\lambda)$ is simply 
\begin{equation}
M(\lambda)=\frac{1}{\mu(X)}\int_{X}K_{\lambda}(x,x)d\mu(x).
\end{equation}Again, the $\gamma=e$ term in (8.13) contributes exactly $c_{3}(\lambda-1)^{3/2}$ to (8.15), so once again the problem is to estimate (8.14) by $O(\lambda)^{1/2}$. Of course, even an estimate $o(\lambda^{3/2})$ would be interesting in this case, since we don't have a Weyl asymptotic law to fall back on.

\textbf{Example 8.5: Compact nilmanifolds}

Let $\widetilde{X}=\heis_{n}$ as in example 3.4, and let $\Gamma$ be a cocompact discrete subgroup of $\heis_{n}$. Note that we are dealing with the subLaplacian on $X=\widetilde{X}/\Gamma$. For simplicity we will take $\Gamma=\ZZ^{2n}\times\frac{1}{2}\ZZ$ in $\CC^{n}\times\RR$. Then
\begin{equation}
K_{\lambda}((z,t),(z,t))=\sum_{\zeta\in\ZZ^{2n}}\sum_{\ell=-\infty}^{\infty}\sum_{\epsilon=\pm1}\sum_{k=0}^{\infty}\varphi_{s,k,\epsilon}\left(\zeta,\frac{\ell}{2}-\imm \zeta\overline{z}\right)d\zeta
\end{equation}where $\varphi_{s,k,\epsilon}$ is given by (3.17).

The choice $(\zeta,\ell)=(0,0)$ gives exactly $\widetilde{M}(\lambda)$ given by (3.21), so it remains to estimate the contribution to (8.16) from all the other choices of $(\zeta,\ell)$. From (3.17) we see this is explicitly the sum of 
\begin{align}
\sum_{k=0}^{\infty}&\int_{0}^{\lambda}\frac{s^{n}}{(2\pi)^{n}(n+2k)^{n+1}}\times\notag\\
&\times2\cos\left(\frac{s}{n+2k}\left(\frac{\ell}{2}-\imm \zeta\overline{z}\right)\exp\left(\frac{-s|\zeta|^{2}}{-4(n+2k)}\right)L_{k}^{n-1}\left(\frac{s|\zeta|^{2}}{2(n+2k)}\right)\right)ds
\end{align}over $(\zeta,\ell)\neq(0,0)$. It seems rather challenging to estimate this sum.



In fact the desired estimates are obtained in [T], p. 79-80 by first studying the heat kernel and then applying a Tauberian Theorem. The heat kernel has the advantage of rapid decay at infinity to avoid convergence problems. A direct approach to these quotients for certain choices of $\Gamma$ is given in [Fo].

It would also be interesting to relate these results to the group representation decomposition of $L^{2}(\heis_{n}/\Gamma)$ in [R].

\textbf{Acknowledgements:} The author is grateful to one of the referees who provided very detailed and helpful comments on the original version of this paper.


\end{document}